\documentclass{amsart}

\usepackage[all]{xy}
\usepackage{amsmath}
\usepackage{hyperref}
\hypersetup{
    colorlinks=true,    
    linkcolor=blue,          
    citecolor=blue,      
    filecolor=blue,      
    urlcolor=blue           
}
\usepackage{tikz}
\usetikzlibrary{graphs,positioning,arrows,shapes.misc,decorations.pathmorphing}

\tikzset{
    >=stealth,
    every picture/.style={thick},
    graphs/every graph/.style={empty nodes},
}

\tikzstyle{vertex}=[
    draw,
    circle,
    fill=black,
    inner sep=1pt,
    minimum width=5pt,
]
\usepackage[position=top]{subfig}
\usepackage[alphabetic,backrefs]{amsrefs}
\usepackage{amssymb}
\usepackage{color}

\calclayout

\usetikzlibrary{decorations.pathmorphing}
\tikzstyle{printersafe}=[decoration={snake,amplitude=0pt}]

\newcommand{\Q}{\mathbb{Q}}
\renewcommand{\qq}{\mathbb{Q}}

\newcommand{\rr}{\mathbb{R}}

\newtheorem{introthm}{Theorem}

\newtheorem{introcor}{Corollary}

\newtheorem{theorem}{Theorem}[section]
\newtheorem{lemma}[theorem]{Lemma}
\newtheorem{proposition}[theorem]{Proposition}
\newtheorem{corollary}[theorem]{Corollary}

\newtheorem{definition}[theorem]{Definition}

\newtheorem{remark}[theorem]{Remark}

\theoremstyle{remark}

\numberwithin{equation}{section}

\begin{document}

\title[Weak Zariski decompositions and termination of flips]{On weak Zariski decompositions and termination of flips}

\author[C.~Hacon]{Christopher Hacon}
\address{
Department of Mathematics, University of Utah, 155 S 1400 E, JWB 233,
Salt Lake City, UT 84112, USA}
\email{hacon@math.utah.edu}

\author[J.~Moraga]{Joaqu\'in Moraga}
\address{
Department of Mathematics, University of Utah, 155 S 1400 E, JWB 321,
Salt Lake City, UT 84112, USA}
\email{moraga@math.utah.edu}


\thanks{The first author was partially supported by NSF research grants no: DMS-1300750, DMS-1265285 
and by a grant from the Simons Foundation; Award Number: 256202. He would also like
to thank the Mathematics Department and the Research Institute for Mathematical Sciences,
located Kyoto University.}

\begin{abstract}
We prove that termination of lower dimensional flips for generalized klt pairs 
implies termination of flips for log canonical generalized pairs with a weak Zariski decomposition.
Under the same hypothesis we prove that the existence of weak Zariski decompositions for pseudo-effective log canonical pairs
implies the existence of weak Zariski decompositions for pseudo-effective generalized log canonical pairs. 
As an application, we prove the termination of any minimal model program for generalized log canonical
pseudo-effective $4$-folds.
\end{abstract}

\maketitle

\setcounter{tocdepth}{1}
\tableofcontents

\section*{Introduction}
One of the main goals of the minimal model program is to show that given a $\qq$-factorial klt pair $(X,B)$ such that $K_X+B$ is pseudo-effective (resp. not pseudo-effective), then there exists a finite sequence of divisorial contractions and flips $$X\dasharrow X_1\dasharrow X_2\dasharrow \ldots \dasharrow X_n$$ such that $(X_n,B_n)$ is a minimal model (resp. there is a Mori fiber space $X_n\to Y$ and in particular $-(K_{X_n}+B_n)$ is ample over $Y$), where $B_n$ is the strict transform of $B$ on $X_n$. We refer the reader to \cite{KM98} for the details of the minimal model program.
After \cite{BCHM}, it is known that the above sequence of flips and divisorial contractions always exists and the only remaining question is wether it terminates after finitely many steps. It is well known that any such sequence can have only finitely many divisorial contractions and hence the main open question is if there are no infinite sequences of flips.
A flip $X\dasharrow X^+$ is a small birational map of $\qq$-factorial varieties, projective over a variety $W$ such that $\rho(X/W)=\rho (X^+/W)=1$ and both $-(K_X+B)$ and $K_{X^+}+B^+$  are ample over $W$ where $B^+$ is the strict transform of $B$.
As a consequence of the negativity lemma, it is easy to see that flips improve certain singularity invariants known as log discrepancies. More precisely, if $X\dasharrow X^+$ is a flip, then we have the following inequality $a_E(X,B)\leq a_E(X^+,B^+)$ which is strict if and only if the center of $E$ is contained in the flipping locus i.e. the exceptional locus of the flipping contraction $X\to W$.
Shokurov has shown \cite{Shok04} that certain natural conjectures concerning log descrepancies can be used to prove
termination of flips in arbitrary dimension.
These conjectures are the ascending chain condition for minimal log discrepancies and the semicontinuity for minimal log discrepancies.
Unluckily these conjectures are very subtle and not well understood in dimension $\geq 3$.
In \cite{BCHM} a different approach is introduced. Instead of trying to prove termination of arbitrary sequences of flips, the authors show termination of specific kinds of minimal model programs known as minimal model programs with scaling. This approach is successful whenever $K_X+B$ is big or $B$ is big or $K_X+B$ is not pseudo-effective. In particular the existence of minimal models for klt pairs of log general type follows as well as the existence of Mori fiber spaces for klt pairs $(X,B)$ such that $K_X+B$ is not pseudo-effective.
This approach does not seem to shed any light on the termination of arbitrary sequences of flips.

In~\cite{Bir07}, Birkar introduced a new philosophy to prove termination of flips for klt pairs such that $K_X+B$ is pseudo-effective. In this case one expects that $K_X+B\equiv G\geq 0$.  Birkar shows  that assuming the  ascending chain condition  conjecture  for log canonical thresholds and the termination of flips for klt pairs of dimension $\leq d-1$, then flips terminate for any $d$-dimensional log canonical pair $(X,B)$ such that $K_X+B\equiv G\geq 0$.
The  ascending chain condition  conjecture for lct's was proved by Hacon, M$^{\rm c}$Kernan and Xu in ~\cite{HMX14}, and later extended to the context of generalized pairs by Birkar and Zhang in ~\cite{BZ16}. In \cite{Shok09}, Shokurov shows that termination of flips with scaling holds for pseudo-effective klt fourfolds and in particular these pairs admit a minimal model and hence a Zariski decomposition. 
In~\cite{Mor18}, the second author proves  termination of pseudo-effective $4$-fold flips by combining the results of \cite{Bir07}, \cite{Shok09} and \cite{BZ16}.
Following this philosophy, in this article we prove that the existence of a weak Zariski decomposition for a generalized log canonical pair 
can be used to reduce termination of flips for such pairs to lower dimensional terminations. More precisely, we prove the following theorem:

\begin{introthm}\label{termination}
Assume termination of flips for generalized klt pairs of dimension at most $ n-1$.
Let $(X/Z,B+M)$ be a generalized log canonical pair of dimension $n$
admitting a weak Zariski decomposition.
Assume that $B$ is a $\qq$-divisor and $M$ is a $\qq$-Cartier b-divisor nef over $Z$.
Then any minimal model program for $K_X+B+M/Z$ terminates.
\end{introthm}

Note in particular that in this paper we work with $\qq$-divisors and our results do not apply to the context of $\rr$-divisors.
We also prove that the existence of weak Zariski decompositions for pseudo-effective generalized log canonical pairs follows from the same statement for generalized log canonical pairs.

\begin{introthm}\label{genwzd}
Assume termination of flips for generalized klt pairs of dimension at most $n-1$.
Then the existence of weak Zariski decompositions for pseudo-effective log canonical pairs of dimension $n$
implies the existence of weak Zariski decompositions for pseudo-effective generalized log canonical pairs of dimension $n$.
\end{introthm}

Note that by Theorem~\ref{genwzd} it follows that in Theorem \ref{termination} it suffices to assume that  log canonical pairs of dimension $n$
admit a weak Zariski decomposition. 
Combining Theorem~\ref{termination}, Theorem~\ref{genwzd}, and the existence of minimal models for
pseudo-effective log canonical $4$-folds~\cite{Shok09}, we prove that any minimal model program for a pseudo-effective generalized log canonical $4$-fold terminates. 
This generalizes the main theorem of~\cite{Mor18} from the case of log canonical pairs
to the case of generalized log canonical pairs with $\qq$-divisors.

\begin{introcor}\label{4-fold-generalized-termination}
Let $(X/Z,B+M)$ be a pseudo-effective generalized log canonical $4$-fold.
Then any minimal model program for $(X/Z,B+M)$ terminates.
\end{introcor}

Again, in this result we assume all divisors have $\qq$-coefficients.\\

\noindent {\bf Acknowledgement.} We would like to thank C. Birkar for useful discussions and suggestions. 
This paper is deeply influenced by his ideas (especially \cite{Bir07} and \cite{Bir12a}). 
We would like to thank J. Han for many useful comments on a previous draft of this paper. 
We would like to thank V. L\'azic for pointing to us a mistake in a previous version of this paper.

\section{Preliminary results}

\subsection{Weak Zariski decomposition}

\begin{definition}\label{wzd}
{\rm Let $D$ be a $\qq$-Cartier divisor on a normal variety $X/Z$. A {\em weak Zariski decomposition for $D$ over $Z$}
consists of a normal variety $X'$, a projective birational morphism $f \colon X' \rightarrow X$, 
and a numerical equivalence 
\[
f^*D \equiv_{Z} P'+N'
\]
such that the following properties hold
\begin{enumerate}
\item $P'$ is a $\qq$-Cartier divisor which is nef over $Z$, and
\item $N'$ is an effective $\qq$-Cartier divisor.
\end{enumerate}
We will say that a generalized pair $(X/Z,B+M)$ (see Definition \ref{d-glp}) has a {\em weak Zariski decomposition} if
the $\qq$-Cartier divisor $K_X+B+M/Z$ has a weak Zariski decomposition.
In what follows, we may write {\em WZD} instead of weak Zariski decomposition in order to shorten the notation.}
\end{definition}

\begin{remark}{\em
Consider a $\qq$-Cartier divisor $D$ on a projective normal variety $X$.
If there exists a projective $D$-non-positive birational contraction $\pi\colon X\dashrightarrow X_1$,
such that the divisorial push-forward $\pi_*D$ is a nef $\qq$-Cartier divisor,  
then $D$ has a weak Zariski decomposition. 
Indeed, we consider a common resolution of singularities with
projective birational morphisms $f\colon X'\rightarrow X$ and $f_1\colon X'\rightarrow X_1$, 
then we can write 
\[
f^*D= f_1^*(\pi_*D)+E,
\]
where $f_1^*(\pi_*D)$ is nef and $E$ is an effective $\qq$-divisor. 
In particular, a pair $(X,B )$ admitting a minimal model has a weak Zariski decomposition.
Therefore, conjecturally, every pseudo-effective log canonical pair has a WZD.}
\end{remark}

\begin{remark}{\em 
In ~\cite{Zar62}, Zariski proved that any effective divisor $D$ on a smooth projective surface $X$
can be decomposed as $P+N$, where $P$ and $N$ are $\qq$-divisors, $P$ is nef, $N$ is effective, 
the intersection matrix of $N$ is negative definite, and $P\cdot C=0$ for every irreducible componente $C$ of $N$.
In ~\cite{Fuj79}, Fujita generalized the above decomposition to the context of pseudo-effective $\rr$-divisors.

There have been many attempts to generalize the above decomposition for  higher dimensional varieties.
For instance, the Fujita-Zariski decomposition ~\cite{Fuj86} and the CKM-Zariski decomposition (see, e.g., ~\cite{Pro04}).
In~\cite{Bir12a}, assuming the minimal model program for dlt pairs in dimension $d-1$, 
the author proves that the existence of a WZD for a log canonical pair of dimension $d$ is equivalent to the existence of all of the above decompositions.}
\end{remark}

\begin{remark}{\em 
In ~\cite{Les14}, the author constructs a psuedo-effective divisor on the blow up of $\mathbb{P}^3$
at nine very general points, which lies in the closed movable cone and has negative intersections with a set
of curves whose union is Zariski dense. Hence, this pseudo-effective divisor does not admit a weak Zariski decomposition. 
}
\end{remark}

\subsection{Generalized pairs}
In this subsection, we recall the language of generalized pairs.

\begin{definition}\label{d-glp}
{\em A {\em generalized pair} is a triple $(X/Z,B+M)$, such that the following conditions hold
\begin{enumerate}
\item $X$ is a quasi-projective normal algebraic variety,
\item $X\rightarrow Z$ is a projective morphism of normal varieties, 
\item $M$ is the push-forward of a $\qq$-divisor, nef over $Z$, on a higher birational model of $X$ over $Z$, 
\item $B$ is an effective $\qq$-divisor,
\item $K_X+B+M$ is a $\qq$-Cartier divisor.
\end{enumerate}
More precisely, there exists a projective birational morphism $f\colon X' \rightarrow X$ 
from a normal quasi-projective variety $X'$ and a nef $\qq$-Cartier $\qq$-divisor $M'$ such that $M=f_* M'$.
We can define $B'$ via the equation
\[
K_{X'}+B'+M'= f^*(K_X+B+M).	
\]
We will say that $B$ is the {\em boundary part}
and $M$ is the {\em nef part} of the generalized pair.
Observe that $M'$ defines a nef b-Cartier $\qq$-divisor in the sense of~\cite[Definition 1.7.3]{Cor07}.
We will say that this is the {\em nef b-divisor  associated to the generalized pair}.}
\end{definition}

\begin{definition}
{\em Given a projective birational morphism 
$g\colon X''\rightarrow X$ which dominates $X'\rightarrow X$, 
we can write 
\[
K_{X''}+B''+M''= g^*(K_X+B+M),
\]
where $M''$ is the pull-back of $M'$ to $X''$.
Given a prime divisor $E$ on $X''$, we define the {\em log discrepancy} of $(X/Z,B+M)$
at $E$ to be 
\[
a_E(X/Z,B+M)= 1-{\rm coeff}_E(B'')
\]
where ${\rm coeff}_E(B'')$ denotes the coefficient of $B''$ along the prime divisor $E$.
We say that $(X/Z,B+M)$ is 
{\em Kawamata log terminal} or {\em klt} 
if the log discrepancy of $(X/Z,B+M)$ at any prime divisor over $X$ is positive,
and we say that $(X/Z,B+M)$ is 
{\em log canonical} or {\em lc}
if the log discrepancy of $(X/Z,B+M)$ at any prime divisor over $X$ is non-negative. 

By Hironaka's resolution of singularities we may assume that $X''$ is smooth and $B''$ has simple normal crossing support.
In this case, $(X/Z,B+M)$ is klt (resp. lc) iff ${\rm coeff}(B'')<1$ (resp. ${\rm coeff}(B'')\leq 1$). Here ${\rm coeff}(B'')$ denotes the biggest coefficient of the $\qq$-divisor $B''$.
}
\end{definition}

\begin{definition}
{\em 
Let $(X,B+M)$ be a generalized pair and $(X'',B''+M'')$ any log resolution as above. 
A prime divisor $E$ of $X''$ such that ${\rm coeff}_E(B'')\geq 1$ is called 
a {\em generalized non-klt place} of the generalized pair $(X,B+M)$.
Moreover, if ${\rm coeff}_E(B'')=1$ (resp. ${\rm coeff}_E(B'')>1$) 
then we may call it a {\em generalized log canonical place}
(resp. {\em generalized non-lc place}) of the generalized pair on $X''$.
The image of a generalized non-klt place (resp. generalized log canonical place) on $X$ is called a
{\em generalized non-klt center} (resp. generalized log canonical center) of the generalized pair.
A generalized non-klt center of a generalized pair $(X,B+M)$ is said to be {\em minimal}
if it is minimal with respect to inclusion.
}
\end{definition}

\begin{definition}
{\em Let $(X/Z,B+M)$ be a generalized pair. A {\em weak contraction} $\phi \colon X \rightarrow W$ for the generalized pair
is a projective birational contraction over $Z$, such that $-(K_X+B+M)$ is nef over $W$.
A {\em quasi-flip} of $\phi$ is a projective birational map $\pi \colon X \dashrightarrow X^+$ with a projective birational contraction
$\phi^+\colon X^+\rightarrow W$ over $Z$ such that the following conditions hold
\begin{enumerate}
\item the triple $(X^+,B^{+}+M^{+})$ is a generalized log canonical pair,
\item the $\qq$-Cartier $\qq$-divisor $K_{X^{+}}+B^{+}+M^{+}$ is nef over $W$,
\item the inequality $\phi^{+}_{*}B^{+}\leq \phi_{*}B$ of Weil $\qq$-divisors on $W$ holds, and
\item the nef parts $M$ and $M^{+}$ are the trace of a common nef b-Cartier b-divisor.
\end{enumerate}
As usual, the morphism $\phi$ (resp. $\phi^+$) is called the {\em flipping contraction} (resp. {\em flipped contraction}).
We may call $(X/Z,B+M)$ (resp. $(X^{+}/Z,B^{+}+M^{+})$) the {\em flipping generalized pair} (resp. {\em flipped generalized pair})
when the quasi-flip is clear from the context.  
}
\end{definition}

\begin{definition}
{\em 
A quasi-flip $\pi$ is said to be {\em ample} if $-(K_X+B+M)$ and $K_{X^{+}}+B^{+}+M^{+}$
are ample over $W$,
and at most one of the morphisms $\phi$ and $\phi^+$ is the identity.
Observe that if $\phi^+$ is the identity, then $\phi$ is a divisorial contraction, and vice-versa.
In the above case, the quasi-flip will be called a {\em weak divisorial contraction} and
{\em weak divisorial extraction}, respectively.
The quasi-flip $\pi$ is said to be {\em small} if both $\phi$ and $\phi^+$ are small morphisms.
A {\em generalized flip} is an ample small quasi-flip of relative Picard rank one. 
A {\em generalized divisorial contraction} (resp. {\em divisorial extraction}) is a 
weak divisorial contraction (resp. weak divisorial extraction) of relative Picard rank one.
A generalized flip for a generalized pair which is generalized klt (resp. generalized dlt or generalized lc)
on a neighborhood of the flipping contraction is called a {\em generalized klt flip} (resp. {\em generalized dlt flip} or {\em generalized lc flip}).
}
\end{definition}

\begin{definition}
{\em 
A sequence of quasi-flips for a generalized log canonical pair $(X,B+M)$ is said to be {\em under a set satisfying the DCC}
if the coefficients of all the boundary parts $B_i$ in the sequence of quasi-flips belong to a fixed set satisfying the DCC. Moreover, we say that the sequence is with a {\em fixed boundary divisor}
if the boundary divisor on the flipped pair is the divisorial push-forward of the boundary divisor on the 
flipping pair.
}
\end{definition}

\begin{definition}
{\em
A {\em minimal model program} for $K_X+B+M$ over $Z$, is a sequence 
of flips and divisorial contractions for $K_X+B+M$ over $Z$.
A {\em weak minimal model program} for $K_X+B+M$ over $Z$, 
is a sequence of ample quasi-flips for $K_X+B+M$ over $Z$.}
\end{definition}

The following proposition is well-known to experts (see, e.g. ~\cite[Monotonicity]{Shok04}).

\begin{proposition}\label{monotonicity}
Given an ample quasi-flip $\pi \colon X \dashrightarrow X^+$ for generalized log canonical pairs 
$(X/Z,B+M)$ and $(X^+/Z,B^{+}+M^{+})$ over $Z$, with flipping contraction 
$\phi \colon X \rightarrow W$, and a prime divisor $E$ over $X$, we have that
\[
a_E(X/Z,B+M) \leq a_E(X^+/Z,B^{+}+M^{+})
\]
and the inequality is strict if and only if the center of $E$ on $X$ is contained in the flipping locus union the support of $B-\pi ^{-1}_* B^+$.
\end{proposition}

\begin{definition}
{\em 
Let $(X/Z,B+M)$ be a generalized log canonical pair, 
let $N$ be an effective divisor on $X$ and $P'$ a nef $\qq$-Cartier divisor over $Z$ on $X'$,
such that $N+P$ is $\qq$-Cartier, where $P=f_*P'$.
The {\em generalized log canonical threshold} of $N+P$ with respect to the generalized pair $(X,B+M)$
is defined to be
\[
{\rm lct}((X/Z,B+M)\mid N+P) :=
{\rm sup}\{ \lambda \mid (X/Z, B+ \lambda N + M+ \lambda P) \text{ is generalized log canonical} \},
\]
where the above generalized pair has boundary part $B+\lambda N$ and
nef part $M+\lambda P$.
Observe that the above threshold is a non-negative rational number or $+\infty$, for instance if $P'=f^*P$ and $N=0$.
}
\end{definition}

\begin{remark}
{\em
Given a set of positive real numbers $\Lambda$ satisfying the DCC, we will denote by $\mathcal{B}(\Lambda)$ the set
of generalized boundaries $B+M$,
where the coefficients of $B$ belong to $\Lambda$, and where we can write $M'=\sum \lambda_i M'_i$
where $\lambda_i\in \Lambda$ and $M'_i$ are Cartier divisors nef over $Z$.
In ~\cite[Theorem 1.5]{BZ16}, Birkar and Zhang prove that the set 
\[
{\rm LCT}_n(\Lambda) = \{ {\rm lct}((X/Z,B+M)\mid N+P) \mid B+M\in \mathcal{B}(\Lambda),  N+P \in \mathcal{B}(\Lambda) \text{ and } \dim(X)=n \}
\]
satisfies the  ascending chain condition. Here, we assume that $N+P$ is $\qq$-Cartier so that the definition of log canonical threshold makes sense.
The proof relies on ~\cite{HMX14}, where this result is proved in the case $M'=N'=0$.
In ~\cite{BZ16}, the authors prove the statement by induction in the number of non-trivial coefficients of $M'$ and $N'$.
Note that if $M$ is a nef $\qq$-divisor, then there exists an integer $k$ such that $M_1=kM$ is nef and Cartier and hence $M=\frac{1}{k}M_1$. A similar statement does not hold for $\rr$-Cartier nef divisors.}
\end{remark}

\begin{remark}
{\em If $M=0$, then we will drop the word ``generalized" from the definition.
In this case, we are in the usual setting of log pairs as in~\cites{KM98,HK10}.}
\end{remark}

\subsection{Log canonical threshold with respect to weak Zariski decompositions}
In this subsection, we introduce an invariant for generalized log canonical pairs admitting a weak Zariski decomposition.

\begin{definition}
{\em
Let $(X/Z,B+M)$ be a $\qq$-factorial generalized log canonical pair with a weak Zariski decomposition 
given by the projective birational morphism $f\colon X' \rightarrow X$ over $Z$ and
the numerical equivalence $f^*(K_X+B+M)\equiv_Z N'+P'$. We consider 
$P= f_*P'$ and $N=f_*N'$ as the nef part and boundary part of a generalized boundary, 
and define
\[
{\rm lct}_{{\rm WZD} (f,N+P)}(X/Z,B+M) := {\rm lct}((X/Z,B+M)\mid N+P).
\]
We call this invariant {\em the log canonical threshold of the generalized pair with respect to the weak Zariski decomposition}
or just the {\em lct with respect to the WZD}.
When the weak Zariski decomposition is clear from the context, we will just write
${\rm lct}_{\rm WZD}$ instead of ${\rm lct}_{{\rm WZD}(f,N+P)}.$}
\end{definition}

\begin{remark}
{\em The generalized log canonical threshold with respect to the weak Zariski decomposition 
depends on the chosen WZD and not only on the given generalized pair. 
For instance, every effective divisor linearly equivalent to $K_X+B+M$ gives a different weak Zariski decomposition, 
and different choices of effective divisors give different log canonical thresholds.
The above invariant is uniquely determined by the generalized pair if we choose a decomposition as defined by Nakayama in~\cite[Chapter 3, \S 1.a.]{Nak04}. 
However, the existence of a WZD is a weaker assumption (see, e.g.,~\cite{Bir12a}).}
\end{remark}

\begin{lemma}\label{finiteness}
Let $(X/Z,B+M)$ be a $\qq$-factorial generalized log canonical pair with a weak Zariski decomposition.
The lct with respect to the WZD is finite unless $K_X+B+M$ is nef over $Z$.
\end{lemma}

\begin{proof}
Without loss of generality we may assume that we have a projective birational morphism $f\colon X'\rightarrow X$
such that both nef b-Cartier divisors $P'$ and $M'$ descend to $X'$.
If $N'$ is a non-trivial effective divisor, then the above log canonical threshold is finite, so we may assume it is trivial.
Hence, $f^*(K_X+B+M) \equiv_Z P'$, so $K_X+B+M$ is nef over $Z$.
\end{proof}

\begin{lemma}\label{highermodel}
The lct with respect to the WZD does not change if we replace $X'$ by a higher birational model.
\end{lemma}

\begin{proof}
The generalized log canonical threshold only depends on $(X,B+M)$, the nef b-Cartier divisor $P'$, 
and the effective divisor $N= f_*N'$.
This data is preserved when replacing $X'$ with a higher birational model.
\end{proof}

\begin{lemma}\label{WZDpreserved} Let $(X/Z,B+M)$ be a $\qq$-factorial generalized log canonical pair with a weak Zariski decomposition $f:X'\to X$ such that $f^*(K_X+B+M)\equiv_Z P'+N'$ where $P'$ is nef over $Z$ and $N'\geq 0$. 
If $\pi:X\dasharrow X_1$ is a quasi-flip with fixed boundary that extracts no divisors and $X_1$ is $\qq$-factorial, then $(X_1/Z,B_1+M_1)$ is a generalized log canonical pair with a compatible weak Zariski decomposition.
\end{lemma}
\begin{proof} We may assume that $f_1:X'\to X_1$ is a morphism. We have 
$$P'+N'\equiv _Z f^*(K_X+B+M)\equiv _Zf_1^*(K_{X_1}+B_1+M_1)+E$$ 
where $E\geq 0$ is $f_1$-exceptional.
Note $N'-E \equiv_{X_1} -P'$ is anti-nef over $X_1$
and ${f_1}_*(N'-E) = {f_1}_* N'  \geq 0$. 
It follows by the negativity lemma that $N'_1:=N'-E\geq 0$. 
But then $f_1^*(K_{X_1}+B_1+M_1)\equiv _Z P'+N'_1$ is a compatible weak Zariski decomposition.
\end{proof}

\begin{lemma}\label{nondecreasing}
Let $(X/Z,B+M)$ be a $\qq$-factorial generalized log canonical pair with a weak Zariski decomposition
and 
\[
 \xymatrix@C=2em{
(X/Z,B+M)\ar@{-->}[r]^-{\pi_1} & (X_1/Z,B_1+M_1)\ar@{-->}[r]^-{\pi_2} & (X_2/Z,B_2+M_2) \ar@{-->}[r]^-{\pi_3} & 
\cdots \ar@{-->}[r]^-{\pi_i} & (X_i/Z,B_i+M_i)\ar@{-->}[r]^-{\pi_{i+1}} & \cdots
 }
\]
a sequence of small ample $\qq$-factorial quasi-flips with fixed boundary for $K_X+B+M$ over $Z$.
Then, the lct of the generalized pairs $(X_i/Z,B_i+M_i)$ with respect to the WZD induced by Lemma \ref{WZDpreserved}
forms a non-decreasing sequence of positive rational numbers. 
\end{lemma}

\begin{proof}
Since $\pi_i$ is a small ample quasi-flip over $Z$
we know that the generalized log canonical pair $(X_i/Z,B_i+M_i)$ is not nef over $Z$.
Hence, by Lemma~\ref{finiteness}, we conclude that the lct with respect to any WZD of $K_{X_i}+B_i+M_i$ over $Z$ is finite.
It suffices to prove the statement for a single small ample quasi-flip $\pi \colon X \dashrightarrow X^+$ over $Z$,  
of the $\qq$-factorial generalized log canonical pair $(X/Z,B+M)$.
We will denote by $(X^{+}/Z,B^{+}+M^{+})$ the flipped generalized log canonical pair.
Consider two projective birational morphisms $f\colon X'\rightarrow X$ and $f^+ \colon X'\rightarrow X^+$ over $Z$,
such that both nef b-Cartier divisors $P'$ and $M'$ descend on $X'$.
We will denote by $f^*(K_X+B+M)\equiv_Z P'+N'$ the induced weak Zariski decomposition for
$K_X+B+M$ on $X'$.
By the negativity lemma we have 
\[
{f^+}^*(K_{X^{+}}+B^{+}+M^{+}) \equiv_Z P' + {N'}^{+}
\]
where $N' \geq {N'}^{+} \geq 0$.
Hence, we have an induced Zariski decomposition for $K_{X^{+}}+B^{+}+M^{+}/Z$
and we will denote 
\[
P^{+}=f^+_*P'\qquad  \text{ and }\qquad  N^{+}=f^{+}_*{N'}^{+}.
\]
Without loss of generality we may assume that $X'$ is a log resolution of both generalized pairs.
By Lemma~\ref{highermodel}, this assumption does not change the lct with respect to the WZD. 
Therefore, by Proposition~\ref{monotonicity} we conclude that for every $\lambda >0$ we have that
\[
{f^+}^*(K_{X^+}+B^{+}+M^{+} +\lambda(P^{+}+N^{+})) \leq
f^*(K_X+B+M+\lambda(P+N)),
\]
concluding the inequality between log canonical thresholds.
\end{proof}

\begin{corollary}
The lct with respect to the WZD of a small ample quasi-flip $(X/Z,B+M)\dashrightarrow (X^{+}/Z,B^{+}+M^{+})$
strictly increases if and only if the flipping locus contains all the generalized log canonical centers of 
$(X/Z,B+M+\lambda(P+N))$ where $\lambda$ is the log canonical threshold
of the generalized pair $(X/Z,B+M)$ with respect to the WZD.
\end{corollary}

\subsection{Generalized divisorially log terminal modifications}
In this subsection, we recall the proof of the existence of $\qq$-factorial 
dlt modifications for generalized log canonical pairs 
(see, e.g.,~\cite[Lemma 4.5]{BZ16} and~\cite[2.13.(2)]{Bir17}).
In ~\cite{AH12} and ~\cite[Theorem 3.1]{KK10}, there is a proof of existence of dlt modifications for pairs.

\begin{definition}\label{d-dlt}
{\em 
We say that the pair $(X/Z,B+M)$ is {\em divisorially log terminal} or {\em dlt} if there exists an open subset $U\subset X$ such that
\begin{enumerate}
\item the coefficients of $B$ are less than or equal to one, 
\item $U$ is smooth and $B|_U$ has simple normal crossings, 
\item all the generalized non-klt centers of $(X,B+M)$ intersect $U$ and are given by strata of $\lfloor B\rfloor$.
\end{enumerate}
}
\end{definition}
Note that if $(X/Z,B+M)$  is dlt and $\phi: X\dasharrow X'$ is a step of the $(X/Z,B+M)$-MMP then $(X',B'+M'=\phi _*(B+M))$ is also dlt.
This can be checked easily by letting $U'=\phi(U\setminus {\rm Ex}(\phi))$ and observing that since $(X,B+M)$ is generalized log canonical, then the flipped locus contains no generalized log canonical centers of $(X',B'+M')$.

\begin{lemma}\label{dlt-perturbation}
Let $(X/Z,B+M)$ be a generalized dlt pair and $U\subset X$ an open subset as in the definition above.
Let $A$ be a divisor on $X$ ample over $Z$ and $\epsilon >0$ a rational number. 
Then, there exists a $\qq$-divisor $B_\epsilon \sim_\qq B+\epsilon A$ so that $(X/Z,B_\epsilon+M)$ is generalized klt.
\end{lemma}

\begin{proof}
Let $\delta>0$ be a sufficiently small rational number.
For any rational  number $0<\delta \ll 1$, we may pick an integer $m>0$ such that $m(A-\delta \lfloor B\rfloor)$ is integral and generated. In particular $| m(A-\delta \lfloor B\rfloor)|$ is base point free on $U$. Let $D\in | m(A-\delta \lfloor B\rfloor)|$ be a general element and set $B_\epsilon =B-\epsilon \delta \lfloor B\rfloor +\frac \epsilon m D$ so that $K_X+B_\epsilon+M\sim _\Q K_X+B+\epsilon A +M$ is $\Q$-Cartier. Since 
the only log canonical centers of $(X/Z,B+M)$ are strata of $\lfloor B\rfloor$ and the support of $D$ contains no such strata, it follows easily that $(X/Z,B_\epsilon+M)$ has no non-klt centers and hence is generalized klt.
\end{proof}

\begin{lemma}\label{l-qf} Let $(X,B+M)$  be a generalized dlt pair, then there exists a small birational morphism $\nu:X'\to X$ such that 
$X'$ is $\Q$-factorial and $(X',B'+M'=\nu ^{-1}_*(B+M))$ is dlt.
\end{lemma}
\begin{proof} By Lemma~\ref{dlt-perturbation}, we may assume that $(X ,B_\epsilon +M)$ is a generalized klt pair. Let $\nu :X'\to X$ be a log resolution and write $K_{X'}+B'_\epsilon +M'=\nu^*(K_X+B_\epsilon +M)+F$ where $(X',B'_\epsilon +M')$ is klt $F\geq 0$ and the support of $F$ equals the sum of all $X'\to X$ exceptional divisors. By  \cite[Section 4]{BZ16}, we may run the $({X'},B'_\epsilon +M')$ MMP over $X$. Replacing $X'$ by the output of this MMP, 
we may assume that $F$ is nef and hence $F=0$ by the negativity lemma. But then $X'\to X$ is a small birational morphism. It remains to show that $(X',B'+M'=\nu ^{-1}_*(B+M))$ is dlt. Let $U$ be the open subset given by Definition \ref{d-dlt}. Since $U$ is smooth, we may assume that $U'=\nu ^{-1}(U)\to U$ is an isomorphism (since any small birational morphism of $\Q$-factorial varieties is an isomorphism). The claim now follows since if $E$ is a divisor over $X'$ with center contained in $X'\setminus U'$, then its center on $X$ is contained in $X\setminus U$ and hence $a_E(X,B+M)>0$. 
\end{proof}

\begin{definition}
{\em
Let $(X/Z,B+M)$ be a generalized pair.
Let $h\colon Y \rightarrow X$ be a projective birational morphism of normal varieties over $Z$. 
We may assume that the given projective birational morphism $f\colon X'\rightarrow X$ factors through $h$.
Then, we define $B_Y$ and $M_Y$ to be the push-forwards of $B'$ and $M'$ on $Y$, respectively.
Thus, we can write
\[
K_Y+B_Y+M_Y=h^*(K_X+B+M).
\]
If the following conditions are satisfied:
\begin{itemize}
\item $B_Y$ is an effective divisor, 
\item $(Y/Z,B_Y^{\leq 1}+M_Y)$ is $\qq$-factorial dlt, where $B_Y^{\leq 1}=B_Y\wedge {\rm Supp}(B_Y)$, and
\item every $h$-exceptional prime divisor $E$ has log discrepancy less than or equal to zero with respect
to the generalized pair $(X/Z,B+M)$,
\end{itemize}
then we say that $(Y/Z,B_Y+M_Y)$ is a {\em $\qq$-factorial dlt modification} of $(X/Z,B+M)$.
Here, we consider $(Y/Z,B_Y+M_Y)$ as a generalized pair with nef b-Cartier divisor $M'$.
Observe that $(Y/Z,B_Y+M_Y)$ is the usual $\qq$-factorial generalized dlt modification over 
the generalized log canonical locus of $(X/Z,B+M)$.
}
\end{definition}

We will prove the following proposition in Section 2.

\begin{proposition}\label{dltmodification}
Assume termination of flips for generalized klt pairs of dimension at most $n-1$.
Let $(X/Z,B+M)$ be a generalized pair of dimension $n$, 
where $B$ is a $\qq$-divisor and $M$ is a $\qq$-Cartier b-divisor nef over $Z$.
Then, $(X/Z,B+M)$ 
has a $\qq$-factorial dlt modification $(Y/Z,B_Y+M_Y)$.
\end{proposition}

The following lemma is proved in a more general setting in ~\cite[Section 4]{BZ16}.

\begin{lemma}\label{mmpdltmodel}
Let $(Y/Z,B_Y+M_Y)$ be a $\qq$-factorial generalized dlt pair.
Let $A$ be a general effective ample divisor on $Y$ over $Z$,
then we can run a minimal model program for the generalized pair with scaling of $A$ over $Z$.
\end{lemma}

\subsection{Generalized dlt adjunction}\label{S-gad}
In this subsection, we recall the construction and properties of generalized divisorial adjunction in ~\cite{BZ16}
and introduce a generalized dlt adjunction formula.

\begin{definition}\label{genadj}
Let $(X/Z,B+M)$ be a generalized log canonical pair, 
assume that $S$ is the normalization of a component of $\lfloor B \rfloor$ 
and $S'$ its birational transform on $X'$.
Replacing the morphism $f\colon X'\rightarrow X$ with a higher birational model,  
we may assume that $f$ is a log resolution for the generalized log canonical pair $(X,B+M)$.
Then, we can write
\[
K_{X'}+B'+M'= f^*(K_X+B+M),
\]
and 
\[
K_{S'}+B_{S'}+M_{S'}= (K_{X'}+B'+M')|_{S'},
\]
where $B_{S'}=(B-{S'})|_{S'}$ and $M_{S'}\sim_\rr M|_{S'}$.
We have an induced morphism $f_S \colon S'\rightarrow S$
and we let ${f_S}_*(B_{S'})=B_S$ and ${f_S}_*(M_{S'})=M_S$.
Hence, we can consider the pair $(S/Z,B_S+M_S)$ as a generalized pair
with b-nef b-Cartier divisor $M_{S'}$.
\end{definition}

\begin{lemma}
The divisor $B_{S}$ is effective. The generalized pair $(S/Z,B_S+M_S)$ is generalized log canonical.
\end{lemma}

\begin{proof}
This is proved in~\cite[Remark 4.8]{BZ16}.
\end{proof}

\vspace{0.05cm}

\begin{proposition}~\label{divadj}
Let $d$ be a natural number and let $\Lambda$ be a set of nonegative rational numbers satisfying the DCC.
There is a set of nonegative rational numbers $\Omega$ satisfying the DCC, which only depends on $d$ and $\Lambda$,
such that if 
\begin{enumerate}
\item $(X/Z,B+M)$ is generalized log canonical of dimension $d$, 
\item the coefficients of $B$ belong to $\Lambda$,
\item we can write $M'=\sum \mu_i M'_i$, where $M'_i$ are Cartier divisors nef over $Z$ with $\mu_i \in \Lambda$, and
\item the generalized pair $(S/Z,B_S+M_S)$ is constructed as in Definition~\ref{genadj},
\end{enumerate}
then the coefficients of $B_S$ belong to $\Omega=\Omega (\Lambda, d)$.
\end{proposition}

\begin{proof}
This is proved in~\cite[Proposition 4.9]{BZ16}.
\end{proof}

\begin{lemma}\label{adjunction}
Let $\Lambda$ be a set of nonegative rational numbers satisfying the DCC condition and $d \in \mathbb{Z}_{\geq 1}$.
Then there is a set of nonegative rational numbers $\Theta$ satisfying the DCC condition, 
which only depends on $d$ and $\Lambda$, such that if 
\begin{enumerate}
\item $(Y/Z,B_Y+M_Y)$ is a generalized dlt pair of dimension $d$,
\item the coefficients of $B_Y$ belong to $\Lambda$,
\item we can write $M'=\sum \mu_i M'_i$, where $M'_i$ are Cartier divisors nef over $Z$ with $\mu_i\in \Lambda$, and
\item $V$ is a generalized log canonical center of $(Y/Z,B_Y+M_Y)$,
\end{enumerate}
then we can write an adjunction formula
\[
(K_Y+B_Y+M_Y)|_V \sim_\rr K_V+B_V+M_V,
\]
where $(V/Z,B_V+M_V)$ is a generalized dlt pair, the coefficients of $B_V$ belong to $\Theta$
and we can write $M'_V=\sum \mu_i M'_{i,V}$, where $M'_{i,V}$ are Cartier divisors and $\mu_i\in \Lambda$.
\end{lemma}

\begin{proof}
We proceed by induction on the codimension of the log canonical center.
If the log canonical center has codimension one, then this is Proposition~\ref{divadj}.
If the log canonical center $V$ has higher codimension, then $V$ is contained in some divisor $S$ which appears with coefficient one in $B$. 
Therefore, by Proposition~\ref{divadj} we can do a divisorial generalized adjunction to $S$.
We claim that $(S/Z,B_S+M_S)$ is generalized dlt and $V$ is a generalized non-klt center of this generalized pair.
Indeed, there is an open set $U\subset X$ so that $U_S:=U\cap S$ is smooth, $B_Y|_{U_S}$ has simple normal crossing, 
all the generalized non-klt centers of $(S/Z,B_S+M_S)$ intersect $U_S$, and these centers are given by strata of $\lfloor B_S \rfloor$.
In particular, $V$ is an intersection of a non-empty set of components of $\lfloor B_S \rfloor$, 
hence it is a generalized non-klt center.
Thus, by the induction hypothesis on the codimension, we can write an adjunction formula
\[
(K_S+B_S+M_S)|_V \sim_\rr K_V+B_V+M_V,
\]
which induces an adjunction formula for $(Y/Z,B_Y+M_Y)$.
\end{proof}

\begin{remark}
Observe that the set $\Theta$ of Lemma~\ref{adjunction} 
is 
\[
\Omega(\Omega( \dots (\Omega(\Lambda,d),d-1),\dots,2),1),
\]
where $\Omega$ is the set of Proposition~\ref{divadj}.
\end{remark}

\begin{corollary}
If $V$ is a minimal non-klt center of the generalized dlt pair $(Y,B_Y+M_Y)$,
then the induced generalized pair $(V,B_V+M_V)$ is generalized klt.
\end{corollary}
\begin{lemma}\label{l-g}
Let $\phi: X\dasharrow X^+$ be a flip for a generalized $\qq$-factorial dlt log pair $(X/Z,B+M)$. Assume that $V$ is a generalized non-klt center and $\phi$ is an isomorphism at its generic point $\eta _V$ such that the induced map $\psi : V\dasharrow V^+=:\phi _* V$ induces an isomorphism of generalized log pairs $(V/Z,B_V+M_V)\cong (V^+/Z,B_{V^+}+M_{V^+})$ on an open subset $V^0\subset V$. Here 
\[
K_V+B_V+M_V=(K_X+B+M)|_V  \text{ and } K_{V^+}+B_{V^+}+M_{V^+}=(K_{X^+}+B^+ +M^+)|_{V^+}
\] 
are induced by adjunction. Then $\phi$ is an isomorphism on a neighborhood of $V^0$ in $X$.

\end{lemma}
\begin{proof} Let $f:X\to W$ be the flipping contraction,  $X'$ be the normalization of the main component of $X\times _Z X^+$ and $p:X'\to X$, $q:X'\to X^+$ the projections, then $p^*(K_X+B+M)=q^*(K_{X^+}+B^++M^+)+E$ where $E\geq 0$ and  ${\rm Supp }(E)=p^{-1}({\rm Ex}(f))$. 
The inclusion $\subset$ is clear. Suppose that $x\in p^{-1}({\rm Ex}(f))$, and $F$ is a divisor with center $x$, then $a_F(X,B+M)<a_F(X^+,B^++M^+)$ (as $\phi$ is a flip and the center of $F$ is contained in the flipping locus). On the other hand,
if $x$ is not contained in the support of $E$, then $p^*(K_X+B+M)=q^*(K_{X+}+B^++M^+)$ in a neighborhood of $x\in X'$ and so $a_F(X,B+M)=a_F(X^+,B^++M^+)$. Therefore $x$ is contained in the support of $E$ as required.

Abusing notation, we also denote $p:V'\to V$ and $q:V'\to V^+$ where $V'$ is the strict transform of $V$ (note that $p:X'\to X$ is an isomorphism around the generic point of $V$).
We have $p^*(K_V+B_V+M_V)=q^*(K_{V^+}+B_{V^+}+M_{V^+})+E|_{V'}$. If $\psi$ is an isomorphism of log pairs on $V^0$, then 
$E|_{V'\cap p^{-1}V^0}=0$ so that $V^0\cap {\rm Ex}(f)=\emptyset$ and hence $\phi$ is an isomorphism on a neighborhood of $V^0$.
\end{proof}

\section{Weak Zariski decompositions and termination of flips} 

\subsection{WZD and termination of flips}

\begin{lemma}\label{mmplcc}
Let $(Y/Z, B_Y+M_Y)$ be a $\qq$-factorial dlt pair and 
\[
 \xymatrix@C=2em{
(Y/Z,B_Y+M_Y)\ar@{-->}[r]^-{\pi_1} & (Y_1/Z,B_{Y_1}+M_{Y_1})\ar@{-->}[r]^-{\pi_2} & (Y_2/Z,B_{Y_2}+M_{Y_2}) \ar@{-->}[r]^-{\pi_3} & 
\cdots 
 }
\]
be a minimal model program which is an isomorphism at the generic point of a log canonical center $V$ of $(Y/Z,B_Y+M_Y)$.
Then, the induced sequence of birational maps (see \S \ref{S-gad})
\begin{equation}\label{mmponV}
 \xymatrix@C=2em{
(V/Z,B_V+M_V)\ar@{-->}[r]^-{\pi_1} & (V_1/Z,B_{V_1}+M_{V_1})\ar@{-->}[r]^-{\pi_2} & (V_2/Z,B_{V_2}+M_{V_2}) \ar@{-->}[r]^-{\pi_3} & 
\cdots 
 }
\end{equation}
is a sequence of ample quasi-flips or identities for the generalized dlt pair $(V/Z,B_V+M_V)$.
\end{lemma}

\begin{proof}
This is proved in~\cite[Proposition 4.3]{Mor18} for the divisorial generalized adjunction.
The general case follows by induction on the codimension of the log canonical center.
\end{proof}

\begin{lemma}\label{tercod1}
Consider a sequence of ample quasi-flips for a generalized klt pair.
Assume that the coefficients of the boundary divisors which appear
in this sequence belong to a set satisfying the DCC.
Then, the sequence of quasi-flips terminate in codimension one, i.e.,
after finitely many ample quasi-flips, all the quasi-flips are small.
\end{lemma}

\begin{proof}
This is proved in~\cite[Lemma 4.26]{Mor18}.
\end{proof}

\begin{corollary}\label{small}
If $V$ is a minimal non-klt center of $(Y/Z,B_Y+M_Y)$ not contained in any of the flipping loci, then 
the sequence of birational transformations~\eqref{mmponV} is eventually a sequence of isomorphisms and small ample quasi-flips
with a fixed boundary divisor and a common b-nef divisor.
\end{corollary}

\begin{lemma}\label{fromqftoflip}
Let $X\dasharrow X^+$ be a small ample quasi-flip over $W$ for generalized klt pairs $(X/Z,B+M)$ and $(X^+/Z,B^++M^+)$ with a fixed boundary divisor.  
Let $(Y/Z,B_Y+M_Y)$  be a small $\qq$-factorialization of $(X/Z,B+M)$. 
Then there exists a sequence of $(Y/Z,B_Y+M_Y)$ flips $\eta: Y\dasharrow Y^+$ over $W$, such that $(Y^+/Z,B_{Y^+}+M_{Y^+})$ is a $\qq$-factorialization of $(X^+/Z,B^++M^+)$.
In particular a small ample quasi-flip for $\qq$-factorial generalized klt pairs with a fixed boundary divisor can be factored in a sequence of flips.
\end{lemma}

\begin{proof}
Suppose that $\pi :X\dasharrow X^+$ is a small ample quasi-flip so that we have  projective morphisms $\phi :X\to W$ and $\phi^+:X^+\to W$ over $Z$ such that $-(K_X+B+M)$ and $K_{X^+}+B^++M^+$ are ample over $W$ and $B^+=\pi _* B$.  
By assumption $\mu :Y\to X$ is a small birational morphism, $Y$ is $\qq$-factorial and $K_Y+B_Y+M_Y=\mu ^*(K_X+B+M)$.
We now run a $K_Y+B_Y+M_Y$ minimal model program with scaling over $W$ 
which terminates by~\cite[Lemma 4.4]{BZ16}. The output of this minimal model program is a good minimal model $(Y^+,B_{Y^+}+M_{Y^+})$ for $K_Y+B_Y+M_Y$ over $W$,
it has a projective birational morphism $\nu\colon Y^+\rightarrow X^+$ such that $K_{Y^+}+B_{Y^+}+M_{Y^+}=\nu^*(K_{X^+}+B^++M^+)$.
\end{proof}

\begin{remark}\label{r-1} Note that ${\rm Ex}(Y\dasharrow Y^+)=\mu ^{-1}{\rm Ex}(X\dasharrow X^+)$.
To see this note that since $X\dasharrow X^+$ is an ample quasi-flip, then ${\rm Ex}(X\dasharrow X^+)$ coincides with the set of points on $X$ that are centers for a divisor $E$ such that $a_E(X,B+M)<a_E(X^+,B^++M^+)$. Similarly since $Y\dasharrow Y^+$ is a sequence of flips, then ${\rm Ex}(Y\dasharrow Y^+)$ coincides with the set of points on $Y$ that are centers for a divisor $E$ such that $a_E(Y,B_Y+M_Y)<a_E(Y^+,B_{Y^+}+M_{Y^+})$. The claim now follows easily since $a_E(X,B+M)=a_E(Y,B_Y+M_Y)$ and $a_E(X^+,B^++M^+)=a_E(Y^+,B_{Y^+}+M_{Y^+})$ for any divisor $E$ over $X$. 
\end{remark}
The following lemma is a version of Fujino's special termination for dlt pairs in the context of generalized pairs (see, e.g.,~\cite{Fuj07}).

\begin{lemma}\label{lemma}
With the notation of Lemma~\ref{mmplcc}. Assume that a minimal model program for the generalized $\qq$-factorial dlt pair $(Y/Z,B_Y+M_Y)$ is infinite.
Then, this minimal model program is eventually disjoint from the generalized non-klt locus of $(Y/Z,B_Y+M_Y)$ or
it induces an infinite sequence of flips for a generalized klt pair of dimension at most $n-1$.
\end{lemma}

\begin{proof}
Assume that the flipping loci of the minimal model program for $(Y/Z,B_Y+M_Y)$ intersect the generalized non-klt locus infinitely many times. Then there exists a generalized log canonical center which is not contained in any exceptional locus of the minimal model program and intersects the flipping loci infinitely many times. Let $V$ be a generalized log canonical center which is minimal with the above condition. By the minimality assumption, eventually the flipping loci only intersect the klt locus of $(V/Z,B_V+M_V)$. Since the generalized pair $(V/Z,B_V+M_V)$ is generalized dlt by Lemma~\ref{adjunction},  then by Lemma \ref{l-qf} it has a generalized $\qq$-factorialization $(V'/Z, B_{V'}+M_{V'})$ and $h:V'\to V$. By Lemma~\ref{mmplcc}, Corollary~\ref{small}, and Lemma~\ref{fromqftoflip}, we obtain an induced infinite sequence of flips for the generalized klt pair $(V'/Z,B_{V'}+M_{V'})$.
\end{proof}

\begin{proof}[Proof of Proposition~\ref{dltmodification}]
Let $(X'/Z,B'+M')$ be a log resolution (given by a sequence of blow ups along smooth centers) of the generalized pair $(X/Z,B+M)$
and denote by $\pi \colon X'\rightarrow X$ the induced birational morphism.
Consider the generalized pair 
\[
(X'/Z, \overline{B} + M'),
\]
where $\overline{B}=\pi ^{-1}_*B+{\rm Ex}(\pi)$ is the effective divisor obtained from $B'$ by 
setting the coefficients of all exceptional divisors over $X$ equal to one.
We claim that the diminished base locus of $K_{X'}+\overline{B}+M'$ over $X$
contains all exceptional divisors over $X$ whose log discrepancy with respect to $(X/Z,B+M)$ is positive.
Indeed, we may write 
\[
K_{X'}+\overline{B}+M' = \pi^*(K_X+B+M)+E_1 - E_2, 
\]
where $E_1$ (resp. $E_2$) is an effective divisor which is supported on the union of all exceptional divisors over $X$ whose
log discrepancy with respect to $(X/Z,B+M)$ is positive (resp. negative).
We can pick an ample divisor $A$ on $X$ and an effective divisor $F$, exceptional over $X$ and supported on ${\rm Ex}(\pi)$ 
so that $\pi^*A-\lambda F$ is ample on $X'$ for $\lambda$ small enough. 
For $\epsilon>0$ arbitrary small the diminished base locus over $X$ of $K_{X'}+\overline{B}+M'$ equals the stable base locus over $X$ of 
\begin{equation}\label{diminished-base-locus}
\pi^*(K_X+B+M+\epsilon A) + E_1- \epsilon \lambda F - E_2. 
\end{equation}
Let $E$ be an effective divisor which is $\qq$-linearly equivalent over $X$ to the $\Q$-divisor from equation~\eqref{diminished-base-locus}.
We have that 
\[
E-E_1+\epsilon\lambda F +E_2 \sim_{X,\qq} 0
\]
and the push-forward to $X$ of the above divisor is effective.
Then by the negativity lemma we have that
\[
E+E_2 \geq E_1  - \epsilon \lambda F.
\]
For $\epsilon$ small enough, any prime divisor contained in the support of $E_1$
appears with positive coefficient in $E_1 -\epsilon\lambda F$,
moreover since $E_2$ and $E_1$ have no common prime components, 
we conclude that for $\epsilon$ small enough the support of $E$ must contain the support of $E_1$, concluding the claim.
By~\cite[4.4]{BZ16}, we may run a minimal model program with scaling of an ample divisor for $(X'/Z,\overline{B}+M')$ over $X$. 
This minimal model program eventually contracts all divisors in the diminished stable base locus and hence all exceptional divisors over $X$
whose log discrepancy with respect to $(X/Z,B+M)$ is positive. Therefore we may assume that $E_1=0$ and thus eventually every flip is $-E_2$ negative and in particular intersects the support of $\lfloor \overline B\rfloor$.
Thus, this MMP terminates by Lemma~\ref{lemma} and the lower dimensional termination of generalized flips.
The obtained minimal model over $X$ is a $\qq$-factorial generalized dlt modification of $(X/Z,B+M)$.
\end{proof}

\begin{proof}[Proof of Theorem~\ref{termination}]
Assume termination of flips for generalized klt pairs of dimension at most $n-1$.

First we prove the case in which $(X/Z,B+M)$ is $\qq$-factorial generalized dlt.
Let $(X/Z,B+M)$ be a $\qq$-factorial generalized dlt pair of dimension $n$ admitting a weak Zariski decomposition.
We proceed by contradiction. Let 
\[
 \xymatrix@C=2em{
(X/Z,B+M)\ar@{-->}[r]^-{\pi_1} & (X_1/Z,B_1+M_1)\ar@{-->}[r]^-{\pi_2} & (X_2/Z,B_2+M_2) \ar@{-->}[r]^-{\pi_3} & 
\cdots \ar@{-->}[r]^-{\pi_i} & (X_i/Z,B_i+M_i)\ar@{-->}[r]^-{\pi_{i+1}} & \cdots
 }
\]
be an infinite minimal model program for $(X/Z,B+M)$.
We denote by $P_{i}$ and $N_{i}$ the push-forward 
of the nef part and effective part of the weak Zariski decomposition induced by Lemma~\ref{WZDpreserved}
on each generalized pair $(X_i/Z,B_i+M_i)$.

Truncating the above infinite minimal model program, 
we may assume that all the steps are flips.
We claim that there exists a non-negative rational number $\lambda$, such that 
the above sequence of generalized flips is a sequence of generalized flips for the generalized pair 
$(X/Z,B+M+\lambda(P+N))$ which is eventually disjoint from every generalized non log canonical center
of $(X/Z,B+M+\lambda(P+N))$ and there exists a generalized log canonical center of $(X/Z,B+M+\lambda(P+N))$
which is intersected by infinitely many generalized flips, but is never contained in the flipping loci.
Indeed, let 
\[
\lambda_0 :={\rm glct}((X/Z,B+M) \mid N+P).
\]
If all generalized flips are eventually disjoint from the generalized log canonical centers
of $(X_i/Z,B_i+M_i+\lambda_0 (N_i+P_i))$, then we define $\lambda_1$ to be the log canonical threshold
of $(X_i/Z,B_i+M_i)$ with respect to $N_i+P_i$ on the complement of such generalized log canonical centers.
Proceeding inductively, we create an increasing sequence of generalized log canonical thresholds, 
which must stop by the ACC for glct's~\cite[Theorem 1.5]{BZ16}.
Hence, eventually we find $\lambda$ as required.
Observe that by Lemma~\ref{finiteness} we may assume that $\lambda$ is always finite. If this were not the case, then there is an open subset $U_i\subset X_i$ containing all the flipping loci such that $(K_{X_i}+B_i+M_i)|_{U_i}\sim _\Q \lambda (N_i+P_i)|_{U_i}$ where  $N_i|_{U_i}=0$, $P_i|_{U_i}$ is nef over $W_i$ and $X_i\to W_i$ is the flipping contraction. But then the flipping contraction $X_i\to W_i$ is $K_{X_i}+B_i+M_i$ trivial, which is impossible.

Up to reindexing our sequence we have an infinite sequence of generalized flips 
\[
 \xymatrix@C=1em{
(X/Z,B+M+\lambda(P+N))\ar@{-->}[rr]^-{\pi_1}\ar[rd] & & (X_1/Z,B_1+M_1+\lambda(P_1+N_1))\ar@{-->}[r]^-{\pi_2}\ar[ld] \ar[rd]&  &\cdots \\ 
& W & & W_1 & \cdots 
 }
\]
for the generalized pair $(X/Z,B+M+\lambda(P+N))$ such that all generalized flips are disjoint from the generalized
non-log canonical centers of $(X/Z,B+M+\lambda(P+N))$ and there exists a generalized log canonical center 
of $(X/Z,B+M+\lambda(P+N))$ which intersects non-trivially infinitely many flipping loci.

If $\lambda=0$, then we obtain a contradiction by Lemma~\ref{lemma}.
Assume that $\lambda>0$ and let 
$(Y/Z,B_Y+M_Y+\lambda(P_Y+N_Y))$ be a $\qq$-factorial dlt modification of $(X/Z,B+M+\lambda(P+N))$ given by Lemma~\ref{dltmodification}.
Here, $M_Y$ (resp. $P_Y$) are the trace of the corresponding b-divisors, 
$\lambda N_Y$ is the strict transform of $\lambda N$, 
and $B_Y$ is the strict transform of $B$ plus the reduced exceptional divisor.
Note that the divisor $K_Y+B_Y+M_Y+\lambda(P_Y+N_Y)$ is anti-nef over $W$.
This statement follows from the proof of Lemma~\ref{dltmodification}.
We denote the morphism of this $\qq$-factorial dlt modification by $\rho\colon Y \rightarrow X$.
We claim that $Y$ is of Fano type over a neighborhood $U$ of the image on $W$ of the flipping locus. 
Indeed, observe that for $0<\lambda' < \lambda$ the generalized pair 
$(X/Z,B+M+\lambda'(P+N))$ is generalized klt on a neighborhood of the flipping locus of $X\rightarrow W$
and anti-ample over $W$. 
We conclude by~\cite[2.10]{Bir17} that $Y$ is of Fano type over an open set $U$ on $W$ which contains the image of the flipping locus.
In particular, we can run a minimal model program for any divisor on $Y$ over $U$ which will terminate with a minimal model.
Now, we run a minimal model program for 
$(Y/Z,B_Y+M_Y+\lambda(P_Y+N_Y))$ over such neighborhood
which terminates with a minimal model.
Observe that the generalized non-log canonical locus of $(Y/Z,B_Y+M_Y+\lambda(P_Y+N_Y))$
is disjoint from the diminished base locus of $K_Y+B_Y+M_Y+\lambda(P_Y+N_Y)$ over $W$.
We conclude that such minimal model program is also a minimal model program 
for $(Y/Z,B_Y+M_Y+\lambda(P_Y+N_Y))$ over $W$ which terminates with a good minimal model
$(Y_1/Z,B_{Y_1}+M_{Y_1}+\lambda(P_{Y_1}+N_{Y_1}))$ over $W$,
and its ample model is $(X_1/Z,B_1+M_1+\lambda(P_1+N_1))$.
Proceeding inductively, we obtain an infinite sequence of generalized dlt flips for
$\qq$-factorial generalized pairs
\[
 \xymatrix@C=1em{
(Y/Z,B_Y+M_Y+\lambda(P_Y+N_Y))\ar@{-->}[rr]^-{\pi_{Y,1}}\ar[d]^-{\rho} & & (Y_1/Z,B_{Y_1}+M_{Y_1}+\lambda(P_{Y_1}+N_{Y_1}))\ar@{-->}[r]^-{\pi_{Y,2}}\ar[d]^-{\rho_1} \ar[r]&\cdots \\ 
(X/Z,B+M+\lambda(P+N))\ar@{-->}[rr]^-{\pi_1}\ar[rd] & & (X_1/Z,B_1+M_1+\lambda(P_1+N_1))\ar@{-->}[r]^-{\pi_2}\ar[ld] \ar[rd]&  \cdots \\ 
& W & & W_1 & \cdots 
 }
\]
Moreover, every such generalized flip is disjoint from the prime divisors which appear with coefficient larger than one in $B_Y+\lambda N_Y$. 
For simplicity, we will write $B'_Y := (B_Y+\lambda N_Y)^{\leq 1} = (B_Y+\lambda N_Y) \wedge {\rm Supp}(B_Y+\lambda N_Y)$
and use the analogous notation for all $Y_i$. 
We obtain an infinite sequence of generalized flips for the $\qq$-factorial dlt pairs
\[
 \xymatrix@C=1em{
(Y/Z,B'_Y+M_Y+\lambda N_Y)\ar@{-->}[rr]^-{\pi_{Y,1}} & & (Y_1/Z,B'_{Y_1}+M_{Y_1}+\lambda N_{Y_1})\ar@{-->}[r]^-{\pi_{Y,2}} \ar[r]&  \cdots \\ 
 }
\]
Since $(X/Z,B+M+\lambda(P+N))$ has a generalized log canonical center which is intersected non-trivially by
infinitely many flipping loci, we conclude that $(Y/Z,B'_Y+M_Y+\lambda N_Y)$ has a generalized log canonical
center which is intersected non-trivially by infinitely many flipping loci.
This is impossible by Lemma~\ref{lemma}.

Now, we prove the general case. 
 Let 
\[
 \xymatrix@C=2em{
(X/Z,B+M)\ar@{-->}[r]^-{\pi_1} & (X_1/Z,B_1+M_1)\ar@{-->}[r]^-{\pi_2} & (X_2/Z,B_2+M_2) \ar@{-->}[r]^-{\pi_3} & 
\cdots \ar@{-->}[r]^-{\pi_i} & (X_i/Z,B_i+M_i)\ar@{-->}[r]^-{\pi_{i+1}} & \cdots
 }
\]
be an infinite minimal model program for the generalized log canonical pair $(X/Z,B+M)$.
By Lemma~\ref{dltmodification} we can take a dlt modification
$(Y/Z,B_Y+M_Y)$ of $(X/Z,B+M)$.
By Lemma~\ref{mmpdltmodel}, we can run a minimal model program for
the $\qq$-factorial generalized dlt pair $(Y/Z,B_Y+M_Y)$ with scaling of a general ample divisor over $W$.
Observe that $(Y/Z,B_Y+M_Y)$ is big over $W$, since the morphism is birational,
hence it has a weak Zariski decomposition over $Z$.
Thus, all the conditions of the $\qq$-factorial dlt case hold, 
and hence this minimal model program terminates with a minimal model $(Y_1/Z,B_{Y_1}+M_{Y_1})$.
Proceeding analogously with the other steps of the minimal model program,
we obtain an infinite minimal model program for $\qq$-factorial generalized dlt pairs, 
\[
 \xymatrix@C=2em{
(Y/Z,B_Y+M_Y)\ar@{-->}[r]^-{\pi_1} & (Y_1/Z,B_{Y_1}+M_{Y_1})\ar@{-->}[r]^-{\pi_2} & \cdots
 }
\]
We claim that $(Y/Z,B_Y+M_Y)$ has a weak Zariski decomposition over $Z$.
Indeed, $(X/Z,B+M)$ has a weak Zariski decomposition over $Z$, so there exists a projective birational morphism
$f\colon X'\rightarrow X$, and a numerical equivalence
\[
f^*(K_X+B+M)\equiv_Z P'+N',
\]
where $P'$ is a $\qq$-Cartier divisor which is nef over $Z$, and $N'$ is an effective $\qq$-Cartier divisor.
Without loss of generality, we may assume that $X'$ dominates $Y$ with a morphism $f_Y\colon X'\rightarrow Y$.
Since $\rho^*(K_X+B+M)=K_Y+B_Y+M_Y$, we conclude that
\[
f_Y^*(K_Y+B_Y+M_Y)\equiv_Z P'+N',
\]
so $(Y/Z,B_Y+M_Y)$ has a weak Zariski decomposition as well.
Hence, all the conditions of the $\qq$-factorial dlt case hold.
Thus, this minimal model program must terminate by the $\qq$-factorial dlt case again.
\end{proof}

\begin{proof}[Proof of Theorem~\ref{genwzd}]
We proceed by induction on the dimension. 
Let $(X/Z,B+M)$ be a generalized pair. Throughout this proof we will work over $Z$. Assume that $K_X+B+M$ is a $\qq$-Cartier pseudo-effective (over $Z$) divisor. Passing to a dlt modification, we may assume that $(X,B+M)$ is $\qq$-factorial generalized dlt and in particular $(X,0)$ is klt.
If $K_X+B$ is pseudo-effective, then by assumption there exists a birational morphism $f\colon X'\rightarrow X$, such that $f^*(K_X+B)=P'+N'$ where $P'$ is nef and $N'$ is effective, 
and $f^*M=M'+E$, where $M'$ is nef and $E$ is effective. Thus, we may write
\[
f^*(K_X+B+M) = (P'+M') + (N'+E).
\]
Therefore, we may assume that $K_X+B$ is not pseudo-effective and $K_X+B+M$ is a $\qq$-Cartier pseudo-effective divisor.
We will follow the arguments of~\cite{Gong15}.
Notice that the pseudo-effective threshold $\lambda$ for $K_X+B$ with respect to $M$ is rational.
This follows from the proof of~\cite[Proposition 8.7]{DHP13} where~\cite[Conjecture 8.2]{DHP13} is replaced by~\cite[Theorem 1.6]{BZ16}.
Replacing $M$ with $\lambda M$ we may assume that $K_X+B+(1-\epsilon)M$ is not pseudo-effective for $0< \epsilon\ll 1$.
By the proof of~\cite[Proposition 8.7]{DHP13} (see~\cite[Lemma 3.1]{Gong15} and~\cite[\S 4]{BZ16}) there exists a birational contraction 
$\phi \colon X\dashrightarrow X_0$ and a projective morphism $f_0 \colon X_0 \rightarrow Z_0$ where $(X_0,B_0+M_0)$ is generalized log canonical, 
$\dim(X_0)>\dim(Z_0),$ $ \rho(X_0/Z_0)=1$, $K_{X_0}+B_0+M_0 \equiv _{Z_0} 0,$ and $M_0$ is ample over $Z_0$.
Passing to a higher model we may assume that $\phi$ and $\phi_0 = f_0 \circ \phi$ are morphisms.
We claim that $K_X+B+M$ admits a a weak Zariski decomposition over $Z_0$.
Note that the numerical Kodaira dimension of the restriction of $K_X+B+M$ to the general fiber $F$ of $\phi_0$ is zero.
To see this notice that as $K_X+B+M$ is pseudo-effective, so is  $(K_X+B+M)|_F$. On the other hand it is easy to see that $\kappa _\sigma ((K_X+B+M)|_F)\leq \kappa _\sigma ((K_{X_0}+B_0+M_0)|_{F_0})=0$. Here $F_0=\phi (F)$ and $\phi _F=\phi |_F$.
By~\cite{Nak04}, we know that 
\begin{equation}\label{num}
(K_X+B+M)|_F \equiv N_\sigma( (K_X+B+M)|_F) \geq 0, 
\end{equation}
where $N_\sigma$ is defined as in~\cite[Chapter 3, 1.12]{Nak04}. 
Note that $$(K_X+B+M)|_F -\phi _F^*((K_{X_0}+B_0+M_0)|_{F_0})- N_\sigma( (K_X+B+M)|_F)\equiv _{F_0}0$$ is $\phi _F$-exceptional and so
by the negativity lemma, $$(K_X+B+M)|_F=\phi _F^*(K_{X_0}+B_0+M_0)|_{F_0}+N_\sigma( (K_X+B+M)|_F),$$ and in particular $N_\sigma( (K_X+B+M)|_F)$ 
is an effective $\qq$-divisor. 
If $\dim(Z_0)=0$, then the above equation gives us the required weak Zariski decomposition. Otherwise, we may assume that $\dim(Z_0)>0$ and that $K_X+B+M \sim_{\qq ,Z_0} N$ for some effective $\qq$-divisor $N$.
This proves the claim.

We may now run a minimal model program with scaling of a general ample divisor over $Z_0$ for the $\qq$-factorial generalized dlt pair $(X,B+M)$ as in~\cite[\S 4]{BZ16}.
By Theorem~\ref{termination} this minimal model program terminates with a minimal model $(X_1,B_1+M_1)$ over $Z_0$. Since $(X,(1-\delta)B+M)$ is generalized klt for $\delta >0$ and every step of the $K_X+B+M$ MMP is a step of the $K_X+(1-\delta)B+M$ for $0<\delta \ll 1$, it follows that $(X_1,(1-\delta)B_1+M_1)$   is generalized klt for $0<\delta \ll 1$
and hence $(X_1,0)$ is klt. It also follows that $K_{X_1}+B_1+(1-\epsilon)M_1$ is not pseudo-effective over $Z_0$ for any $\epsilon >0$.

We claim that for $\epsilon>0$ small enough, we may run a minimal model program for $K_{X_1}+B_1+(1-\epsilon)M_1$ with scaling of an ample divisor over $Z_0$ such that all steps of this MMP are $(K_{X_1}+B_1+M_1)$-trivial. Pick $r\in \mathbb N$ such that $r(K_{X_1}+B_1+M_1)$ is Cartier and choose a rational number $0<\epsilon<\frac 1{2r\dim(X)} $. We will prove that the first step of this minimal model program is a flop and all the relevant conditions are preserved. Indeed, observe that such a step of the minimal model program must be $(K_{X_1}+B_1)$-negative, hence by~\cite{Kaw91} we have that 
\[
0 < -(K_{X_1}+B_1) \cdot C \leq 2\dim(X)
\]
for some curve $C$ spanning the corresponding extremal ray.
If $(K_{X_1}+B_1+M_1)\cdot C >0$, then $(K_{X_1}+B_1+M_1)\cdot C \geq 1/r$ by the assumption on the Cartier index, 
so we deduce that $(K_{X_1}+B_1+(1-\epsilon)M_1)\cdot C>0$, leading to a contradiction.
Since $K_{X_1}+B_1+M_1$ is nef, we deduce that the above flip must be trivial with respect to this generalized pair. 
Finally, observe that the nefness and the Cartier index of $K_{X_1}+B_1+M_1$ are preserved in this minimal model program which terminates with a Mori fiber space by~\cite{BCHM}.
Therefore, we obtain a $(K_{X_1}+B_1+M_1)$-trivial birational contraction $\pi \colon X_1 \dashrightarrow X_2$, and a $(K_{X_1}+B_1+M_1)$-trivial fiber space $\phi_2 \colon X_2\rightarrow Z_1$ over $Z_0$
which is a Mori-fiber space for $(X_2, B_2+(1-\epsilon)M_2)$. In particular $M_2$ is ample over $Z_1$ and $K_{X_2}+B_2+M_2 \sim _{\qq,Z_1}0$.

 Since $X\dasharrow X_2$ is $K_X+B+M$ non positive, it suffices now to prove that $K_{X_2}+B_2+M_2$ has a weak Zariski decomposition.

By~\cite[Theorem 1.4]{Fil18}, there exists a generalized log canonical pair $(Z_1, B_{Z_1}+M_{Z_1})$ such that
\[
\phi_2^*(K_{Z_1}+B_{Z_1}+M_{Z_1})= K_{X_2}+B_2+M_2.
\]
By induction on the dimension, we may assume that $(Z_1,B_{Z_1}+M_{Z_1})$ has a weak Zariski decomposition	say $h\colon Z_2\rightarrow Z_1$
with
\[
h^*( K_{Z_1}+B_{Z_1}+M_{Z_1})  \equiv P_{Z_2}+N_{Z_2},
\]
where $P_{Z_2}$ is nef and $N_{Z_2}$ is effective. 
Let $\nu\colon X_3 \rightarrow X_2$ be the normalization of the main component of $X_2\times_{Z_1} Z_2$
and write $K_{X_3}+B_3+M_3=\nu^*(K_{X_2}+B_2+M_2)$ where $(X_3,B_3+M_3)$ is the corresponding generalized pair.
Let $\phi _3: X_3\to Z_2$ be the induced morphism.
We have 
\[
K_{X_3}+B_3+M_3 = \nu^*\phi_2^*( K_{Z_1}+B_{Z_1}+M_{Z_1}) \equiv \phi_3^*(P_{Z_2}+N_{Z_2}) = \phi_3^*P_{Z_2}+\phi_3^*N_{Z_2}
\]
which is the desired weak Zariski decomposition.
\end{proof}

\begin{proof}[Proof of Corollary~\ref{4-fold-generalized-termination}]
It is known that every pseudo-effective log canonical $4$-fold has a minimal model (see, e.g.~\cite{Shok09}),
hence every pseudo-effective log canonical $4$-fold has a weak Zariski decomposition.
Moroever, the termination of generalized $3$-fold flips is proved in~\cite[\S 4]{Mor18}.
Hence, by Theorem~\ref{genwzd} we conclude that every pseudo-effective generalized log canonical $4$-fold
has a weak Zariski decomposition.
Thus, by Theorem~\ref{termination} we conclude that any minimal model program for a pseudo-effective
generalized log canonical $4$-fold $(X/Z,B+M)$ terminates.
\end{proof}

\end{document}